\documentclass[11pt]{article}
\usepackage[centertags]{amsmath}
\usepackage{amsfonts}
\usepackage{amssymb}
\usepackage{amsthm}
\usepackage{newlfont}

\newcommand{\comment}[1]{}

\newtheorem{theorem}{Theorem}
\newtheorem{lemma}{Lemma}

\newtheorem{proposition}{Proposition}

\newtheorem{definition}{Definition}

\begin{document}

\title{Perfect simulation of autoregressive models with infinite memory}
\author{ Emilio De Santis \\
{\small Dipartimento di Matematica }\\
{\small La Sapienza Universit\`a di Roma}\\
{\small \texttt{desantis@mat.uniroma1.it}} \and Mauro Piccioni \\
{\small Dipartimento di Matematica }\\
{\small La Sapienza Universit\`a di Roma}\\
{\small \texttt{piccioni@mat.uniroma1.it}}}
\maketitle

\begin{abstract}
In this paper we consider the problem of determining the law of
binary stochastic processes from transition kernels depending on the
whole past. These kernels are linear in the past values of the
process. They are allowed to assume values close to both $0$ and
$1$, preventing the application of usual results on uniqueness. More
precisely we give sufficient conditions for uniqueness and
non-uniqueness. In the former case a perfect simulation algorithm is
also given.
\end{abstract}

\medskip

\medskip \noindent \textbf{keywords:} Perfect simulation, Coupling, Chains
with complete connections.


\section{Introduction and main definitions}

In this paper we consider processes indexed by $ \mathbb{Z}$ taking
values in a finite alphabet $G$. These processes are constructed
from a transition kernel which depends on the whole past, i.e. a map
$p : G \times G^{-\mathbb{N}_+} \to [0,1]$ such that, for any choice
of $\mathbf {w}_{-\infty}^{-1}=(w_{-1},w_{-2},\ldots) \in
G^{-\mathbb{N}_+}$, $p(\cdot| \mathbf {w}_{-\infty}^{-1} )$ is a
probability measure on $G$. For the basic setting of the theory of
these processes we address the reader to
\cite{CFF}  where an historical overview is also given. In
particular we recall the following
\begin{definition}\label{DLR}
A $G$-valued process $ \mathbf{X}=
\{X_n, n \in \mathbb{Z}\} $ defined on an arbitrary probability
space is \emph{compatible}
with the transition kernel $p$ if for any $g \in G$
\begin{equation}\label{compat}
    P( X_n =g |\mathbf{X}_{-\infty}^{n-1})
    =p(g |\mathbf{X}_{-\infty}^{n-1}),
\end{equation}
almost surely, for any $ n \in \mathbb{Z}$.
\end{definition}
It is clear that the compatibility condition
refers only to the law of $\mathbf{X}$, so we are allowed to speak about
compatible laws, whose set will be denoted by $\mathcal{G}(p)$. When $|\mathcal{G}(p)|=1$
the law of any compatible process $\mathbf{X}$ is always the same in which case
we say that uniqueness holds for $p$.

It is immediately seen that $ \mathcal{G}(p)$ is convex. Moreover,
if the kernel $p$ is continuous w.r.t. the product topology in the second coordinate, then $ \mathcal{G}(p)$
is non empty (see
\cite{FM05}). In this paper we will establish sufficient conditions
for uniqueness for a certain class of continuous kernels, giving in
these cases an explicit construction for the unique compatible
process.

In the sequel we will consider binary processes taking values in $G=\{-1,+1\}$.
The transition kernels we are interested in have the form
\begin{equation}\label{specif}
    p^{\mathbf{\theta}}(g|\mathbf {w}_{-\infty}^{-1} )= \frac{1}{2} +
    \frac{g}{2} \sum_{k=1}^\infty \theta_k w_{-k} , \,\,\, g \in G,
    \mathbf {w}_{-\infty}^{-1} \in G^{-\mathbb{N}_+},
\end{equation}
where $\mathbf{\theta}=\{\theta_k, k \in \mathbb{N}_+\}$ is a sequence of constants with the
property $\sum_{k=1}^\infty |\theta_k | =1  $. A kernel of this form is clearly continuous,
since by specifying some large but finite
portion of $\mathbf {w}_{-\infty}^{-1}$, the oscillation of $p^{\mathbf{\theta}}(g|\cdot)$ can
be made arbitrarily small. Moreover, since $ p^{\mathbf{\theta}}(g|\mathbf {w}_{-\infty}^{-1} )
= p^{\mathbf{\theta}}(-g|-\mathbf {w}_{-\infty}^{-1} ) $, $ \mathcal{G}(p^{\mathbf{\theta}})$ is invariant under
the transformation changing the sign to the whole configuration.
Therefore, if $| \mathcal{G}(p^{\mathbf{\theta}})| =1 $ and $ \{X_n\}$ is
the unique compatible process, $P( X_n =\pm 1 ) =1/2 $ for any integer $n \in \mathbb{Z}$.

Binary linear kernels were first introduced in \cite{CFF}, as simple
examples of transition kernels depending on the whole past. More
precisely they added a 
\textquotedblleft bias", that is an extra term $\frac
{\theta_0}{2}$, with $ \theta_0 \in [-1, 1]$; if a bias term is
present the model is well defined provided $\sum_{k=0}^\infty
|\theta_k |\leq 1$. But then they got $\sum_{k=1}^\infty |\theta_k
|< 1$, in which case uniqueness always holds. In fact the unique
compatible process can be constructed by means of a simple perfect
simulation algorithm, to which we will briefly refer in the last
section. Alternatively, as in \cite{FM05}, in this case one can
prove uniqueness by means of Dobrushin's contraction condition.

Thus only when $\theta_0 =0$ the choice $\sum_{k=1}^\infty |\theta_k
|=1$ is allowed. As a consequence the convex hull of the range space of
$p^{\mathbf{\theta}}(1|\cdot)$ (and of
$p^{\mathbf{\theta}}(-1|\cdot)$) is the whole interval $[0,1]$.
Therefore there is no uniform positivity condition for this kind of
kernels, preventing the applicability of general criteria for
uniqueness (see \cite{FM05}). Indeed uniqueness does not always
hold. The simplest situation in which uniqueness does not hold is
when $\theta_k \geq 0$ for all $k \in \mathbb{N}_+$. Then it is easy to
check that the two constant processes constantly equal to $+1$ and
$-1$ are compatible. More generally we can establish the following

\begin{theorem}\label{alterna} Suppose that
\begin{itemize}

\item[a)] either $\theta_k \geq 0$ for all $k \in \mathbb{N}_+$,

\item[b)] or $\theta_k \leq 0$ for $k$ odd and $\theta_k \geq 0$ for $k$
even.
\end{itemize}

Then uniqueness does not hold for $p^{\mathbf{\theta}}$.
\end{theorem}

On the positive side, our main result is the following sufficient condition for uniqueness, which
will be supplemented by a perfect simulation algorithm producing the unique compatible process.
In order to state it, we set
$ A_{\mathbf{\theta}} = \{ k \in \mathbb{N}_+: \theta_k \neq 0\}$.

\begin{theorem}\label{unico}
Suppose that all the following assumptions hold:
\begin{itemize}

\item[i)]  $\gcd(A_{\mathbf{\theta}})=1$;

\item[ii)] there exists either an even integer $m \in \mathbb{N}_+$ with $\theta_m<0$, or an odd
integer $m_1$ with $\theta_{m_1}<0$ and another odd integer $m_2$
with $\theta_{m_2}>0$;

\item[iii)] $\sum_{k=1}^{\infty} (\sum_{m=k}^{\infty} |\theta_m|)^2<+\infty$.
\end{itemize}
Then uniqueness holds for $p^{\mathbf{\theta}}$.
\end{theorem}

We briefly comment the relation among Theorem \ref{alterna} and
Theorem \ref{unico}. Condition i) in Theorem~\ref{unico} is
necessary for uniqueness (Proposition \ref{lemma:k0}); condition ii)
is the negation of assumptions a) and b) of Theorem \ref{alterna}.
At the present time only condition iii)  is preventing the statement
of a set of necessary and sufficient conditions for uniqueness. We
will briefly comment about this within the conclusions.

In Section 2 we will prove Theorem \ref{alterna} and related
results. In Section 3 we will prove Theorem \ref{unico} and related
results. In Section 4 we give the simulation algorithm explicitly
with a few conclusive remarks.

\section{Coherent sequences and nonuniqueness}

The set $A=A_{\mathbf{\theta}}$ allows to define an oriented graph
structure $\Gamma=( \mathbb{Z}, \mathcal{A} ) $ over the set of
integers. There is an arc $ (n,m) \in \mathcal{A}$ if $ m < n $ and
$n-m \in A$. Such an arc is marked by the sign of $\theta_{ n-m}$.
It is clear that, by its very definition, an arc in $\Gamma$ can be
always shifted, since if $ (n,m) \in \mathcal{A}$, then $ (n+k,m+k)
\in \mathcal{A}$, for any $k \in \mathbb{Z}$. An oriented path is a
finite sequence $\{ (n_i , n_{i+1}) \in \mathcal{A}, i = 0,\ldots ,
l-1 \}$. This path joins $n_0$ with $n_l$. Each oriented path has a
sign, which is given by the product $\prod_{i=0}^{l-1}
\hbox{sign}(\theta_{n_i-n_{i+1}})$ of the signs of all its arrows.

By neglecting the orientation of the arrows in $\mathcal{A}$ we
obtain the set of edges $\widetilde{\mathcal{A}}$ of the undirected
graph $\widetilde{\Gamma}$ which is called the symmetrization of
$\Gamma$. In this framework a path is a finite sequence $\{\{n_i ,
n_{i+1}\} \in\widetilde{\mathcal{A}}, i=1,\ldots,l \}$. The sign of
the path is given by $\prod_{i=0}^{l-1}
\hbox{sign}(\theta_{|n_i-n_{i+1}|})$. Here the endpoints $n_0$ and
$n_l$ joined by the path could coincide, but in this case we prefer
to use the word cycle rather than path. By convention for the path
we always take $n_0>n_l$, unless when specified otherwise, but we
say that an edge $\{n_i , n_{i+1}\}$ is traveled in the opposite
direction when $(n_{i+1},n_i) \in \mathcal{A}$, therefore $n_i <
n_{i+1}$. By ordering the edges of the path $\{\{n_i , n_{i+1}\}
\in\widetilde{\mathcal{A}}, i=1,\ldots,l \}$ by decreasing rather
than increasing the index $i$ the same path is traveled in the
direction which is opposite to the previous one. Finally, as arcs
can be shifted in $\mathcal{A}$, a whole path can be shifted in
$\widetilde{\mathcal{A}}$ in a natural way.


By B\'ezout's identity the group generated by $A$ is $\gcd (A)
\times \mathbb{Z} $, and the connected components of
$\widetilde{\Gamma}$ coincide with the congruence classes
$\emph{mod}\gcd (A)$. In particular, when $\gcd (A) =1$ the graph
$\widetilde{\Gamma}$ is connected. This means that for any pair of
sites $n, m \in \mathbb{Z}$ there exists a path joining them.

\begin{proposition} \label{lemma:k0}
The condition $\gcd (A_{\mathbf{\theta}})=1$ is necessary for $ | \mathcal {G}(p^{\mathbf{\theta}})|
=1$.
\end{proposition}

\begin{proof}
Suppose $\gcd(A)=k_0 >1$. Define $\theta^*_k=\theta_{k k_0}$ for $k
\in \mathbb{N}_+$. Let $\mathbf{X}^*$ be any
$p^{\theta^*}$-compatible process. It is easily seen that any
process $\mathbf{X}=\{X_{n}, n \in \mathbb{Z}\}$ such that
$\mathbf{X}^{(h)}=\{X_{nk_0+h}, n \in \mathbb{Z}\}$ has the same law
of $\mathbf{X}^*$, for any $h=0,1,\ldots,k_0-1$, is a
$p^{\mathbf{\theta}}$-compatible process. Since the joint
distribution of $\mathbf{X}^{(h)}$ for $h=0,1,\ldots,k_0-1$ is
arbitrary it is clear that uniqueness does not hold for
$p^{\mathbf{\theta}}$. In particular one can take all
$\mathbf{X}^{(h)}$'s  either equal or independent, in order to
produce two different elements of
$\mathcal{G}(p^{\mathbf{\theta}})$.
\end{proof}

Another source of nonuniqueness comes with the following definition.

\begin{definition}\label{LDR}
A configuration $\mathbf{s} \in G^{\mathbb{Z}}$
is \emph{coherent} with the sequence of coefficients $\mathbf{\theta} \in G^{\mathbb{N}_+}$ if for any
$(n,m) \in \mathcal{A}_{\mathbf {\theta}}$ it is $s_n s_m \theta_{n-m}>0$.
\end{definition}
Coherent configurations come always in pairs, since if $\mathbf{s} \in G^{\mathbb{Z}}$ is coherent, the
same is true for $-\mathbf{s}$.

\begin{proposition}\label{coherentstates} A configuration $\mathbf{s} \in G^{\mathbb{Z}}$
coherent with $\mathbf{\theta} $
is a degenerate compatible process (i.e. $\delta_{\mathbf{s}} \in \mathcal{G}(p^{\mathbf{\theta}})$).
Thus, when a coherent configuration exists, then $|\mathcal{G}(p^{\mathbf{\theta}})|>1$.
\end{proposition}

\begin{proof}
First notice that if $s_n s_m \theta_{n-m}>0$, then 
$s_n s_m \theta_{n-m} =|\theta_{n-m}|$. Now, if $X_n=s_n$, $n \in \mathbb {Z}$,
then, for any $n \in \mathbb{Z}$, it is obviously $P(X_n = s_n |
\mathbf{X}^{n-1}_{-\infty} = \mathbf{s}^{n-1}_{-\infty})=1$. But on
the other hand
\begin{equation}\label{coerente}
     1=\frac{1}{2} +
    \frac{1}{2} \sum_{k=1}^\infty | \theta_k |=
    \frac{1}{2} +
    \frac{s_n}{2} \sum_{k=1}^\infty \theta_k s_{n-k}
\end{equation}
which agrees with (\ref{specif}).
\end{proof}

The following result characterizes the existence of coherent configurations in a convenient way.

\begin{lemma}\label{vecchia} If a coherent configuration exists all the paths joining any two
given sites have the same sign. Conversely, if all the \emph{oriented} paths joining any two
given sites have the same sign, a coherent configuration necessarily exists.
\end{lemma}

\begin{proof}
 First notice that, if a coherent configuration exists, for any path
 $\{ \{n_i , n_{i+1}\} \in \widetilde {\mathcal{A}} \}_{i = 0,\ldots , l-1 }$
 we have $s_{n_0}s_{n_{l}} \prod_{i=0}^{l-1}\theta_{|n_i-n_{i+1}|}>0$, from which necessarily all other
 paths joining $n_0$ and $n_{l}$ must have the same sign. Conversely, choose a representative for
 each connected component of $\widetilde{\Gamma}$, and define arbitrarily the configuration $\mathbf{s}$ over
 these sites. Suppose $n_0$ is one of these representatives and consider any path
 $\{ \{n_i , n_{i+1}\} \in \widetilde {\mathcal{A}} \}_{i = 0,\ldots , l-1 }$. By defining
 \begin{equation}\label{coerente}
 s_{n_{l}}=s_{n_0}\prod_{i=0}^{l1}\hbox{sign}(\theta_{|n_i-n_{i+1}|}),
 \end{equation}
 one can extend unambiguously $\mathbf{s}$ to a coherent configuration. It
 remains to prove that if all the oriented paths between any two sites have
 the same sign, the same is true for all the paths. In fact, suppose that there are two
 paths with different signs joining two distinct sites. By following one of the two paths
 in the opposite direction and concatenating it with the first we get a cycle.
 By assumption the sign of this cycle is negative.
 We can now produce two oriented paths joining the same endpoints by shifting and concatenating the arcs which are
 traveled in the same direction. These two oriented paths necessarily have opposite signs.
 \end{proof}

 By the previous result, when $\gcd(A)=1$ and all the paths between any two sites have the
 same sign, a coherent configuration $\mathbf{s}$ is uniquely determined by the choice of the sign
 $s_0$ at the origin. Next observe that by shifting a coherent configuration, another coherent
 configuration is obtained. Now if a configuration is invariant under the shift, it is necessarily constant.
 If it changes sign under the shift, it is necessarily one of the
 two \textquotedblleft checkerboard" configurations alternating
 periodically $-1$ with $+1$. Thus we have the following result.

\begin{proposition}\label{coefficienti}

When $\gcd(A_{\mathbf{\theta}})=1$ a configuration coherent with $p^{\mathbf{\theta}}$ exists
only in the following two cases:

\begin{itemize}

\item[a)]  either $\theta_k \geq 0$, for all positive integers $k$;

\item[b)] or $\theta_k \leq 0$ for $k$ odd and $\theta_k \geq 0$ for $k$
even.
\end{itemize}
\end{proposition}

\begin{proof}
The two cases correspond to the two possible pairs of coherent
configurations, the constants and the checkerboards. It is easily
checked that the constant configurations are coherent in case a), and
the checkerboards in case b).
\end{proof}

From Propositions \ref{lemma:k0}, \ref{coherentstates} and
\ref{coefficienti}, Theorem \ref{alterna} immediately holds.

\medskip

Before turning to discuss sufficient conditions for uniqueness we
observe that any pair of configurations $\pm \mathbf{s} \in
G^{\mathbb{Z}}$ is equivalently described by an unordered partition
of $\mathbb{Z}$ into two subsets, which are the inverse images of
the two possible values $-1$ and $+1$.

If $\gcd(A)=1$ and the checkerboard configurations are coherent,
this partition has the property that each pair
of sites with the same parity is joined by a positive path, and each
pair of sites with different parity is joined by a negative path. If
$\gcd(A)=1$ and the constant configurations are coherent, then we
have a degenerate partition with a single class: all sites are
joined by a positive path. This interpretation will be
useful when dealing with partitions of this kind for suitable
subgraphs of $\widetilde {\Gamma}$.

\section{Random subtrees and uniqueness}

We start by observing that sampling from $p^{\mathbf{\theta}}$ given in (\ref{specif}) is
realized by drawing a random variable
$K$ with $P(K=k)=|\theta_k|$, for $k=1,2,\ldots$ and returning $X^*=\hbox{sign}(\theta_K)w_{-K}$.
Indeed, by setting $y_k=\hbox{sign}(\theta_k)w_{-k}$, one has

\begin{equation}\label{gauge}
    P(X^*=1|\mathbf {w}_{-\infty}^{-1} )=\sum_{k:y_k=+1}|\theta_k|=
    \frac {1}{2}\{\sum_{k:y_k=+1}|\theta_k|+(1-\sum_{k:y_k=-1}|\theta_k|)\}=
    \frac {1}{2}+\frac{1}{2}\sum_{k=1}^{\infty}|\theta_k|y_k,
\end{equation}
and since $|\theta_k|y_k=\theta_k w_{-k}$, the statement is proved.

Thus, given a $p^{\mathbf{\theta}}$-compatible process $\mathbf {X}=\{X_n, n \in
\mathbb{Z}\}$ we have the following resampling scheme, that does not
change its distribution. Given any site $n \in \mathbb{Z}$, we
select $K_n$ from the distribution $\{\theta_k, =1,2,\ldots\}$, and
then pick the arc $(n,n-K_n)$ from the set of arcs in
$\mathcal{A}$ outcoming from $n$. The resampled value $X^*_n=\hbox {sign} (\theta_{K_n}) X_{n-K_n}$ is
then obtained by sending the value at the sink $n-K_n$ of the
chosen arc back to the source $n$, changing the sign if the
sign of the arc is negative.

By drawing $K_n$ as above, independently for each $n \in
\mathbb{Z}$, we obtain in this way a random subgraph
$\Gamma_r=(\mathbb{Z}, \mathcal{B})$ of $\Gamma=(
\mathbb{Z},\mathcal{A} ) $, and we call its symmetrization
$\widetilde{\Gamma}_r=(\mathbb{Z}, \widetilde{\mathcal{B}})$. This
will play a fundamental role for uniqueness. Since in $\mathcal{B}$
there is only one outcoming arc from each site $n \in \mathbb{Z}$,
in $\widetilde{\mathcal{B}}$ there is no cycle: thus the connected
components of $\widetilde{\Gamma}_r$ are trees, and such a graph is
a forest.

Let us define
\begin{equation}\label{filtr}
  \mathcal{F}_k^m
  = \sigma ( K_{ i }: k\leq i\leq m),
\end{equation}
where $-\infty \leq k <m\geq +\infty$. Define the left tail
$\sigma$-algebra $\mathcal {F}_{-\infty}= \cap_{m \in \mathbb{Z}}
\mathcal {F}^m_{-\infty}$. The number $N$ of trees which forms
$\widetilde{\Gamma}_r$ is a random variable is $\mathcal
{F}_{-\infty}$--measurable since it is possible to read $N$ from the
values $\{K_n, n\leq m\}$ no matter how far the site $m \in
\mathbb{Z}$ lies in the past. In fact all these trees necessarily
\textquotedblleft have to be born at $-\infty$", since each oriented
path in $\mathcal {B}$ can be extended infinitely to its left. Since
the $K_n$'s are i.i.d., it follows that $N$ is a.s. constant. In the
following we will prove that provided the assumptions of Theorem
\ref{unico} do not hold (otherwise coherent sequences exist) and
$N=1$ a.s. (which clearly implies $\gcd(A)=1$), there is uniqueness
for $\mathbf {\theta}$.

Before proving this result we exhibit what in these cases will turn
out to be the unique compatible process. Since $N=1$ any two sites
will have a unique path connecting them. Hence we can partition the
graph $\widetilde{\Gamma}_r$ in two disjoint components in the
following way. Starting from an arbitrary site, say the origin, we
define the random set of sites $S$ which are joined to the origin by
a positive path (including the origin itself). Each pair of sites
belonging to this set is joined by a unique positive path. The same
property holds for the complement $S^c$. Moreover, each site in $S$
is connected to any site in $S^c$ by a negative path. In order to
define a process from this partition it remains to select a random
sign common to sites in $S$ (the opposite one is given to those in
$S^c$). This is achieved by assigning to the origin one of the two
signs with equal probability. To sum up we define
\begin{equation}\label{camp}
X_n=Y(2 \cdot \mathbf{1}_S(n)-1), n \in \mathbb{Z}
\end{equation}
where $Y$ is independent of $S$, with distribution $P( Y =\pm 1 ) =1/2 $.

\begin{proposition}\label{compa}
The process $\mathbf{X}=\{X_n, n \in \mathbb{Z}\}$ defined in (\ref{camp}),
is compatible with $p^{\mathbf{\theta}}$.
\end{proposition}

\begin{proof}
It is clear that, since the origin is connected to any site almost
surely, also the sign of such a site will have the same symmetric
distribution. Since the unordered partition $\{S,S^c\}$ is invariant
by translation of the $K_n$'s, the process $\{X_n, n \in
\mathbb{Z}\}$ is certainly stationary. Thus, to verify
compatibility, it is sufficient to prove that resampling the value
in $1$ according to (\ref{specif}) will not change its distribution.
For this we substitute $K_0$ with $K_0^*$ having the same
distribution, independently of all the $K_n$'s. This means that the
subtree attached to the origin, which originally is attached to the
site $-K_0$, becomes attached to the site $-K_0^*$. The resulting
partition is the one which would be obtained directly from $\{K_n, n
\neq 0, K_0^*\}$, hence it has the same distribution as before.
Moreover the sign of the origin is either that of $X_{-K^*_0}$ or
the opposite, depending on the sign of $\theta_{K^*_0}$. But the
sign of $X_{-K^*_0}$ has again a symmetric distribution,
independently of all the $\{K_n, n \neq 0, K_0^*\}$, because of the
way $X_0=Y$ was defined. This ends the proof.
\end{proof}

For proving our uniqueness result it is now necessary to introduce
the processes $\mathbf{X}^{-n,\mathbf{w}}$, parameterized by $n \in
\mathbb{Z}$ and $\mathbf{w} \in G^{-\mathbb {N}_+}$, with the
\textquotedblleft boundary conditions"
$X_{-n-i}^{-n,\mathbf{w}}=w_{-i}$ for $i=1,2,\ldots$ and for all the
remaining sites $m\geq -n$ obtained by following the unique oriented
path $\{(N_i,N_{i+1}), i=0,\ldots, L_m-1\}$ starting from $N_0=m$ in
$\Gamma_r$ (which means that $N_{i+1}=N_i-K_{N_i}$), stopped when it
\textquotedblleft reaches a boundary site" (which means that
$L_m=\inf\{i:N_i<-n\}$). Finally
\begin{equation}\label{strage}
X_{m}^{-n,\mathbf{w}}=w_{N_{L_m}+n}\prod_{i=0}^{L_m-1} \hbox{sign}(\theta_{N_i-N_{i+1}}),
\end{equation}
so that $\mathbf{X}^{-n,\mathbf{w}}$ is obtained from the boundary conditions by following
iteratively the sampling scheme introduced at the beginning of the section. All these processes are
thus constructed on the probability space where the process $\{K_n, n \in \mathbb{Z}\}$ lives.

Here is the promised result.

\begin{theorem}\label{unicodet}
Consider a transition kernel $p^{\mathbf{\theta}}$ of the form (\ref{specif}) with the following
assumptions on the coefficients:
\begin{itemize}

\item[a)]
 there exists either an even integer $m \in \mathbb{N}_+$ with $\theta_m<0$, or
 an odd integer $m_1$ with $\theta_m<0$ and another odd integer $m_2$ with $\theta_n>0$;

\item[b)] the random graph $\widetilde {\Gamma}_r$ is a.s.
connected.

\end{itemize}
\end{theorem}

In order to prove the above result, we prove first a very useful lemma.

\begin{lemma} \label{probinzero}
Under conditions a) and b)
\begin{equation}\label{equiprob}
    \lim_{n \to \infty} \sup_{\mathbf{w} \in
    G^{\mathbb{N}_+}} \left |P(X_0^{-n,\mathbf {w}}=1)-\frac{1}{2} \right |=0
\end{equation}
\end{lemma}

\begin{proof}

First observe that assumption a) ensures that there are no coherent
configurations. By Lemma~\ref{vecchia} we know that this requires
the existence of two oriented paths in $\widetilde {\Gamma}$ with
the same endpoints, say $0$ and $-m_0<0$, but with different sign.
By excluding a common initial subpath to these paths, one can reduce
$m_0$ to the minimum possible value. Next denote the positive path
by $\gamma^+$ and the negative one by $\gamma^-$. We write $p_+$
(resp. $p_-$) for the probability that $ \gamma^+$ (resp. $ \gamma^-
$) occurs, an event which is $\mathcal {F}^0_{-\infty}$--measurable.

Let us introduce the random times $\{T_i, i \in \mathbb{N}_+\}$, defined by recursion from $T_0=0$.
These random times are a subset of the endpoints of the arcs of the infinite oriented path $\{(N_i,N_{i+1}), i
\in \mathbb{N}\}$
in $\Gamma_r$, starting from $N_0=0$. Whereas
$N_i-N_{i+1}=K_{N_i}$, for $i \in \mathbb{N}$, the $T_i$'s need to have a distance not smaller than
$m_0$, in order to give to $\gamma^+$ and $\gamma^-$ a positive probability to happen. This explains the
following definition.
\begin{equation}\label{renew}
    T_{i+1}=\sup \{m\leq T_i-m_0: \hbox{there is a path in $\Gamma_r$ joining $T_i$ with
    $m$}\}.
\end{equation}
Furthermore, for $i \in \mathbb {N}_+$, define the marks
\begin{equation}\label{berno}
Y_i= \left \{  \begin{array}{ll}
       +1 & \hbox{if $\gamma^+$ joins $T_{i-1}$ with $T_{i}$,}  \\
                 -1 & \hbox{if $\gamma^-$ joins $T_{i-1}$ with $T_{i}$,} \\
                 0 & \hbox{otherwise.} \\
               \end{array}  \right .
\end{equation}
Finally define
$J_n=\sup\{m:T_m \geq -n\}$, $R_n^+=|\{i\leq J_n: Y_i=+1\}|$ and
$R_n^-=|\{i\leq J_n: Y_i=-1\}|$. It is clear that $\lim_{n \to \infty}(R_n^- + R_n^+ )=\infty   $
almost surely. Therefore there exists $M_n \uparrow \infty $ and $
\varepsilon_n \downarrow 0 $ such that $ P( R_n^- + R_n^+ \geq M_n )
\geq 1- \varepsilon_n $. Then
$$
P(X_0^{-n,\mathbf{w}}=+1) \leq  \varepsilon_n + \sum_{k =M_n}^{\infty}
P(X_0^{-n,\mathbf{w}}=+1 |  R_n^- + R_n^+ =k) P( R_n^- + R_n^+ =k) $$
$$
\leq \varepsilon_n + \max_{k\geq M_n } P(X_0^{-n,\mathbf {w}}=+1 |  R_n^- +
R_n^+ =k).
$$
Since, each time that a segment which is a shift of $\gamma^+$ is replaced by
the same shift of $\gamma^-$ the value of $X_0^{-n,\mathbf{w}} $ changes sign
we can bound the previous formula by
\begin{equation}\label{quasi}
    \leq \varepsilon_n + \max_{k\geq M_n }\left \{\max_{  } \{P(R_n^-
\hbox{ is odd}| R_n^- + R_n^+ =k) , P(R_n^- \hbox{ is even}| R_n^- +
R_n^+ =k) \}\right \}.
\end{equation}
In order to bound the previous expression we notice that $R_n^-$
conditioned to the event $\{ R_n^- + R_n^+ =k \}$ has a binomial
distribution, namely $B(k, \frac{p_-}{p_+  +p_-} )$.
For any random
variable $X_k \sim B(k, p )$, with $ p \in (0,1)$, it is easily
proved that $P( X_k \hbox{ is even})$ goes to $1/2$ as $k \to
\infty$. As a consequence we obtain that the expression
(\ref{quasi}) is bounded from above by
$$
 \frac{1}{2} +\varepsilon_n  + \delta_{M_n}
$$
where $\lim_{k \to \infty}\delta_k=0$. Since the same bound can be produced for $P(X_0^{-n,\mathbf{w}}=-1)$,
this ends the proof of the lemma.
\end{proof}

Notice that, since for random variables $X$ and $Y$ taking values in $\{-1,+1\}$ with laws $\mathcal {L}(X)$
and $\mathcal {L}(Y)$ it holds
$$
||\mathcal {L}(X)-\mathcal {L}(Y)||_{TV}=\sup_{||f||_{\infty}\leq 1}|E(f(X))-E(f(Y))|=|P(X=1)-P(Y=1)|,
$$
we can reformulate the above lemma as
$$
\lim_{n \to \infty}\sup_{\mathbf {w}_i \in G^{-\mathbb{N}_+}, i=1,2}||\mathcal {L}(X_0^{-n,
\mathbf {w}_1})-\mathcal {L}(X_0^{-n, \mathbf {w}_2})||_{TV}=0.
$$

\begin{proof}
[Proof of Theorem \ref{unicodet}] In order to prove uniqueness let us consider two
$p^{\mathbf{\theta}}$-compatible processes $\mathbf {X}$ and
$\mathbf {Y}$ defined on the same probability space. By using a sequence $\{K_i, i \in \mathbb{Z}\}$ of i.i.d.
random variables with distribution $\mathbf {\theta}$, independent of $\mathbf
{X}$ and $\mathbf {Y}$, it is possible to sample recursively all the values at sites $m\geq -n$,
according to the sampling scheme described at the beginning of the
section, getting process $\mathbf {X}^*$ and $\mathbf {Y}^*$ that have the same law of $\mathbf {X}$
and $\mathbf {Y}$, and that will be identified with them from now on. Notice that the invariance of the
laws under resampling allows to write
$$
\mathbf {X}=\mathbf {X}^{-n,\mathbf {X}^{-n-1}_{-\infty}}, \mathbf {Y}=
\mathbf {X}^{-n,\mathbf {Y}^{-n-1}_{-\infty}}.
$$

Since $N=1$ a.s., for any finite window of sites $[l,L]=\{i \in
\mathbb{Z}: l\leq i \leq L\}$, with $-\infty <l<L<\infty$, the root
of the smallest subtree of $\widetilde{\Gamma}_r$ containing $[l,L]$
is a proper random variable $T_{[l,L]}$.

For any function
$f:G^{L-l+1}\to \mathbb{R}$, with $||f||_{\infty} \leq 1$ we write
\begin{equation*}
E\{f(\mathbf{X}^L_l)-f(\mathbf{Y}^L_l)\}=E\{1_{\{T_{[l,
L]}>-\frac{n}{2}\}}(f(\mathbf{X}^L_l)-f(\mathbf{Y}^L_l))\}+E\{1_{\{T_{[l,
L]}\leq-\frac{n}{2}\}}(f(\mathbf{X}^L_l)-f(\mathbf{Y}^L_l))\}
\end{equation*}
\begin{equation}\label{ineq2}
\leq E\{1_{\{T_{[l,
L]}>-\frac{n}{2}\}}(f(\mathbf{X}^L_l)-f(\mathbf{Y}^L_l))\}+\frac {\varepsilon}{2}
\end{equation}
for any $\varepsilon>0$, for $n$ sufficiently large, since $T_{[l,
L]}$ is finite a.s.

Observe that $T_{[l,L]}$ is
$\mathcal{F}^L_{T_{[l,L]}+1}$--measurable, and
$$
\mathbf {X}_l^L=X_{T_{[l,L]}}\mathbf {Z}_l^L, \quad \mathbf{Y}^L_l=Y_{T_{[l,L]}}\mathbf
{Z}_l^L,
$$
where $\mathbf {Z}_l^L$ is an $(L-l+1)$-dimensional
$\mathcal{F}^L_{T_{[l,L]}+1}$--measurable random vector: moreover
$X_{T_{[l,L]}}$ and $Y_{T_{[l,L]}}$ are independent of
$\mathcal{F}_{T_{[l,L]}+1}^L$. Therefore
\begin{equation}\label{ineq3}
E\{1_{\{T_{[l,L]}>-\frac{n}{2}\}}(f(\mathbf{X}^L_l)-f(\mathbf{Y}^L_l))\}=
E\{1_{\{T_{[l,L]}>-\frac{n}{2}\}}E(f(\mathbf{X}^L_l)-f(\mathbf{Y}^L_l)|\mathcal{F}^L_{T_{[l,L]}+1})\}.
\end{equation}
Finally, setting $f_{\mathbf {Z}_l^L}(g)=f(g \times \mathbf {Z}_l^L)$, for $g \in
G$, we have
\begin{equation}\label{ineq4}
E(f(\mathbf{X}^L_l)-f(\mathbf{Y}^L_l)|\mathcal{F}^L_{T_{[l,L]}+1})=\mathbb{E}\{f_{\mathbf {Z}_l^L}
(X_{T_{[l,L]}}) -f_{\mathbf {Z}_l^L}(Y_{T_{[l,L]}})\},
\end{equation}
where $\mathbb{E}$ means expectation w.r.t. to the joint law of $X_{T_{[l,L]}}$ and $Y_{T_{[l,L]}}$.
Thus, plugging the above expression into \eqref{ineq3}, we get
\begin{equation*}
E\{1_{\{T_{[l,L]}>-\frac{n}{2}\}}(f(\mathbf{X}^L_l)-f(\mathbf{Y}^L_l))\}=\sum_{-n/2<i<l} E(
1_{\{T_{[l,L]}=i\}}\mathbb{E}(f_{\mathbf {Z}_l^L}(X_{i}) -f_{\mathbf {Z}_l^L}(Y_{i})))
\end{equation*}
\begin{equation}\label{ineq5}
\leq \sum_{-n/2<i<l} P(T_{[l,L]}=i)
\sup_{\mathbf {w}_1,\mathbf {w}_2}||\mathcal {L}(X_i^{-n,\mathbf {w}_1})-
\mathcal {L}(X_i^{-n,\mathbf {w}_2})||_{TV}.
\end{equation}
By Lemma \ref{probinzero} and the subsequent remark the total variation term can be made smaller than $\frac
{\varepsilon}{2}$ as $n$ is large enough. Plugging \eqref{ineq5}
into \eqref{ineq2} we thus obtain that the total variation distance between the laws of
$\mathbf{X}^L_l$ and $\mathbf{Y}^L_l$ does not exceed $\varepsilon$, proving the desired result.
\end{proof}

The natural question is under which conditions on the sequence of
coefficients $\mathbf{\theta}$ i is possible to ensure that $N=1$
a.s. or, equivalently, that any two sites, say e.g. $0$ and $-1$,
are a.s. connected by a path in $\widetilde{\Gamma}_r$. We need to
follow the two infinite oriented paths in $\Gamma_r$ starting from
this two sites to check if they coalesce or not. This can be done by
keeping track recursively in time of the distance between the two
paths, using the rule that it is always the rightmost path which is
updated with the addition of an arc. If the rightmost path reaches
the site $i$, it jumps backwards to $i-K_i$, and so on,
independently of the arcs chosen in the past. In this way the
distance between the two paths follows the Markov chain
\begin{equation}\label{vonSchelling}
Y_{n+1}=|Y_n-K_{n+1}|, n \in \mathbb{N}, Y_0=1
\end{equation}
where the $K_i, i=1,2,\ldots$ are i.i.d. with distribution $P(K_1=j)=\theta_j$. This
process is known in the literature as the von Schelling process
(see \cite{Fell}). It evolves as a random walk with negative
increments until it falls below zero, in which case it is reflected
on the positive side. When $\gcd(\mathcal {A}_{\mathbf {\theta}})=1$, this process is irreducible (see
\cite{Boud}). Moreover we have the following result.

\begin{proposition}\label{marchereflechie}
Consider the random graph $(\mathbb{Z}, \mathcal{B})$ with
$\mathcal{B}=\{\{i,i-K_i\},i \in \mathbb{Z}\}$, where the $K_i, i=1,2,\ldots$ are i.i.d. with
distribution $P(K_1=j)=\theta_j$.
Assume $\gcd({k: \theta_k \neq 0})=1$. Then the number of
connected components $N$ of $(\mathbb{Z}, \mathcal{B})$ is a.s. equal to $1$ if and only if the
corresponding von Schelling process is recurrent.
\end{proposition}

\begin{proof} It is clear that $N=1$ a.s. if and only if the state $0$ is hit a.s.
when the von Schilling process is started in $1$. By irreducibility,
this happens if and only if the process is recurrent.
\end{proof}

An invariant measure for the von Schelling process was written down
in \cite{Boud}. Whenever it has finite total mass then the process is positive recurrent.
This happens whenever $E(K_0)=\sum_{k=1}^{\infty} P(K_0
\geq k)<+\infty$. Recently some null recurrent cases have been covered
by a result in \cite{PW11}, which is reported below.
Incidentally, in the cited paper it is also given a class
of distributions with very heavy tails for which the von
Schelling process is transient.

\begin{theorem}\label{Woess} If $\sum_{k=1}^{\infty} P(K_0 \geq k)^2 <+\infty$,
then the von Schelling process is recurrent.
\end{theorem}

By putting together the two Theorems \ref{unicodet} and \ref{Woess} and Proposition
\ref{marchereflechie} we have finally obtained the promised Theorem \ref{unico}.

We finally remark that  $N =1 $ can be seen as a percolation type
condition. For other examples of percolation techniques applied to
perfect simulation in dimensions higher than one the reader is
referred to \cite{DP}.

\section{Perfect simulation and conclusions}

In the following we give a possible implementation of the simulation of the
unique $p^{\mathbf {\theta}}$-compatible process
in a finite window $[l,L]$.For the sake of simplicity
and without any loss of generality we have set $l=1$.

{\em Simulation algorithm for producing $\mathbf {X}=(X(1), \ldots, X(L))$}

\begin{itemize}

\item[1.] For $i=1,...,L$, set $X(i)=1, S(i)=i, Y(i)=1$ and $Z(i,j)=\delta_{i,j}, j=1,\ldots,L$.
\item[2.] Draw $K$ from the distribution $\mathbf {\theta}$.
\item[3.] Compute $J=\hbox {argmax} \{S(i)Y(i)\}$.
\item[4.] Set $S(J)\leftarrow S(J)-K$ and $X(J) \leftarrow X(J)\times \hbox {sign}(\theta_K)$.
\item[5.] If $S(J)\neq S(i), i\neq J$ go to 2.
\item[6.] Else let $I$ s.t. $S(I)=S(J)$, set $Y(J)=0$ and $Z(J,I)=X(J) \times X(I)$.
\item[7.] If $\sum_{i=1}^LY(i)>1$, go to 2.
\item[8.] Else let $H$ such that $Y(H)=1$, and set $X(i)=0, i \neq H$, drawing $X(H)$ as a symmetric Bernoulli sign.
\item[9.] Output $\mathbf {X} \leftarrow 
Z^L\mathbf {X}$, the product of the matrix $Z^L$ with
the column vector $\mathbf {X}$.

\end{itemize}

A few words of explanation. The vector $\mathbf{S}=(S(1), \ldots,
S(L))$ contains the positions of particles started from the sites of
the interval $[1,L]$, performing independent random walks until
coalescence. These particles will be labeled with their starting
sites. Their displacements are drawn from the distribution
$\mathbf{\theta}$; at each step it is always the rightmost particle
which moves to the left. The vector $\mathbf{X}=(X(1), \ldots,
X(L))$ keeps track of the current sign of the paths followed by the
particles. When particle $J$ lands in a site already occupied by
particle $I$ we set to $0$ the $J$-th value of the auxiliary vector
$\mathbf{Y}=(Y(1), \ldots, Y(L))$. This means that particle $J$
remains \textquotedblleft freezed" in that site and does not play a
further role, aside from the fact that in the auxiliary $L \times L$
matrix $Z$ the $(J,I)$-th element registers if the paths followed by
the two particles until coalescence have or not the same sign. This
agreement or disagreement will be kept in the final vectors of
signs. When the next to the last particle lands in the site
$T_{[1,L]}$ occupied by the last one, say particle $H$, by computing
the $L$-th power of the matrix $Z$ it is possible to reconstruct the
agreement or disagreement of signs between all the particles.
Finally we have slightly modified the construction presented in the
last section by drawing from a symmetric Bernoulli the sign $X(H)$
assigned to the $H$-th particle rather than the sign of a fixed
particle: this clearly will not change the final distribution at
all. All the elements of the desired vector of signs can then be
obtained by multiplying $Z^L$ for a column vector with all zeros
except $X(H)$ in the $H$-th position.

With a few changes, the algorithm can be adapted to the biased case
in which the term $\frac {\theta_0}{2}$ is added to the r.h.s. of
(\ref{specif}). Recall that in this case
$r_0=\sum_{k=1}^{\infty}|\theta_k|<1$ so that the mass $1-r_0$ is
placed on $0$ to get a mass distribution on the integers $\mathbf
{\theta}^*$ which replaces $\mathbf {\theta}$. Now at each step
there is a positive probability that a particle \textquotedblleft
gets stuck" in a site. When the $J$-th particle stops its sign is
sampled from the distribution of a binary random variable $U$ with
$P(U=\pm 1)=\frac {1-r_0\pm \theta_0}{2(1-r_0)}$. This is sampled
independently for particles stuck in different positions. The final
vector of signs is obtained by post-multiplying $Z^L$ for the column
vector containing the signs drawn for the stuck particles and zeros
in the other positions. Such an algorithm is a slight modification,
in order to profit of coalescence, to the one described in
\cite{CFF}.

However, in the latter paper, a general purpose \textquotedblleft
coupling from the past" algorithm was described, based on general
bounds on the \textquotedblleft dependence of the kernel on the
past". For this general algorithm to work,  in the case
$\sum_{k=1}^{\infty}|\theta_k|<1$, a control on the tails of the
distribution $\mathbf {\theta}^*$ is needed, which is quite stronger
than that required by Theorem \ref{Woess}. We should also mention
that in \cite{DP2} a modification of such an algorithm is proposed,
which works also for an example in which $\theta_0=0$ and
$\sum_{k=1}^{\infty}|\theta_k|=1$. But such a modification was
rather ad hoc, and it was not clear how to pursue the same strategy
for a generic vector $\mathbf {\theta}$ with
$\sum_{k=1}^{\infty}|\theta_k|=1$.

It is worth to mention that in the case $N>1$ a.s. it is still
possible to produce a compatible process by repeating independently
for each tree of the graph $(\mathbb{Z}, \mathcal{B})$ the random
assignment of sign to the two sets of the partition described above.
But the uniqueness proof as well as the perfect simulation algorithm
given here obviously fails, even if uniqueness can still hold.

Finally, in dimension higher than one, the reader is referred to
\cite{GLO,DLiss} for perfect simulation of infinite range Gibbs
fields.

\medskip\medskip\medskip\medskip

\bibliographystyle{plain}


\end{document}